\documentclass[12pt,oneside]{article}
\usepackage{amssymb,amsmath}
\usepackage{amsfonts}
\usepackage{mathrsfs}
\usepackage{latexsym}
\usepackage{amssymb}
\usepackage{amsthm}
\usepackage{indentfirst}
\begin{document}

\def\abstractname{\bf Abstract}
\def\dfrac{\displaystyle\frac}
\def\dint{\displaystyle\int}
\def\vec#1{\overset\rightarrow{#1}}
\let\oldsection\section
\renewcommand\section{\setcounter{equation}{0}\oldsection}
\renewcommand\thesection{\arabic{section}}
\renewcommand\theequation{\thesection.\arabic{equation}}
\newtheorem{theorem}{\indent Theorem}[section]
\newtheorem{lemma}{\indent Lemma}[section]
\newtheorem{proposition}{\indent Proposition}[section]
\newtheorem{definition}{\indent Definition}[section]
\newtheorem{remark}{\indent Remark}[section]
\newtheorem{corollary}{\indent Corollary}[section]
\renewcommand{\proofname}{\indent\it\bfseries Proof.}

\title{\LARGE\bf
Global Strong and Weak Solutions to Nematic Liquid Crystal Flow in Two Dimensions
\\
\author{Jinkai Li$^1$
\thanks{Corresponding author. Email: {\it jkli@math.cuhk.edu.hk}}
\\
{\it \small $^1$  The Institute of Mathematical Sciences, The Chinese University of Hong Kong, Hong Kong}
 } }
\date{}

\maketitle

\begin{abstract}
\end{abstract}
We consider the strong and weak solutions to the Cauchy problem of the inhomogeneous incompressible nematic liquid crystal equations in two dimensions. We first establish the local existence and uniqueness of strong solutions by using the standard domain expanding method, and then extend such local strong solution to be a global one, provided the initial density is away from vacuum and the initial direction field satisfies some geometric structure. The size of the initial data can be large. Based on such global existence results of strong solutions, by using compactness argument, we obtain the global existence of weak solutions with nonnegative initial density.

{\bf Keywords}: global strong solutions; global weak solutions; nematic liquid crystal.

\section{Introduction}\label{sec1}

The evolution of liquid crystals in $\mathbb R^d$ is described by the following system
\begin{eqnarray}
&&\partial_t\rho+\textmd{div}(\rho u)=0,\label{1.1}\\
&&\rho(u_t+(u\cdot\nabla)u)-\Delta u+\nabla p=-\textmd{div}(\nabla d\odot\nabla d),\label{1.2}\\
&&\textmd{div}u=0,\label{1.3}\\
&&\partial_t d+(u\cdot\nabla) d=\Delta d+|\nabla d|^2d,\label{1.4}\\
&&|d|=1,\label{1.5}
\end{eqnarray}
where $\rho$ is the density, $u$ represents the velocity field of the flow, $d$ is the unit vector field that represents the macroscopic molecular orientation of the liquid crystal material and $p$ denotes the pressure function. The notation $\nabla d\odot\nabla d$ is a $d\times d$ matrix whose $(i, j)$-the entry is $\partial_id\cdot\partial_jd$, $1\leq i, j\leq d$.

System (\ref{1.1})--(\ref{1.5}) is a simplified version of the
Ericksen-Leslie model, which reduces to the Ossen-Frank model in the
static case, for the hydrodynamics of nematic liquid crystals
developed by Ericksen \cite{E1}, \cite{E2} and Leslie \cite{LE} in
the 1960's. Both the full Ericksen-Leslie model and the simplified
version are the macroscopic continuum description of the time
evolution of the materials under the influence of both the flow
velocity field $u$ and the microscopic orientation configurations
$d$ of rod-like liquid crystals. A brief account of the
Ericksen-Leslie theory and the derivations of several approximate
systems can be found in the appendix of \cite{LL1}. For more details
of physics, we refer the readers to the two books of Gennes-Prost
\cite{GP} and Chandrasekhar \cite{CH}. Though the above system is a
simplified version of the full Ericksen-Leslie system, it still
remains the most important mathematical structures as well as most
of the essential difficulties of the original Ericksen-Leslie
system.

In the homogeneous case, i.e. $\rho\equiv C$, Lin-Lin \cite{LL1,LL2}
initiated the mathematical analysis of (\ref{1.2})--(\ref{1.4}) in
the 1990's. More precisely, they proved in \cite{LL1} the global existence of weak solutions to the initial and boundary value problem to the system (\ref{1.2})--(\ref{1.4}) with $|\nabla d|^2d$ replaced by Ginzburg-Landau type approximation term
$\frac{1-|d|^2}{\varepsilon^2}d$ in bounded domain of two or three dimensions. They also obtain the unique existence of
global classical solutions in dimension two or in dimension three but with large enough viscous coefficients. In \cite{LL2}, they proved the
partial regularity theorem for suitable weak solutions, similar to
the classical theorem by Caffarelli-Kohn-Nirenberg \cite{CKN} for
the Navier-Stokes equation. When take the term $|\nabla d|^2d$ into consideration, the system becomes more complicated from the mathematical point
of view, since it is a supercritical term in the equations of $d$. So far, the global existence of weak solutions are proven only for the two dimensional
case, see Lin, Lin and Wang \cite{Lin4}, Hong \cite{Hong1} and Hong and Xin \cite{Hong2}. The approach used in \cite{Lin4} and that used in
\cite{Hong1,Hong2} are different, where the global existence is proven directly to the liquid crystal system with term $|\nabla d|^2d$, while in
\cite{Hong1,Hong2}, the strategy is to show the convergence of the solutions to the approximate system with penalty term $\frac{1-|d|^2}{\varepsilon^2}d$
as $\varepsilon\rightarrow0$. The uniqueness of such weak solution was later proven
in \cite{LW}.

In non-homogeneous case, i.e. the density dependent case, the global existence of weak solutions to the liquid crystals equations with penalty term $f(d)$ instead of $|\nabla d|^2d$ is established by Jiang and Tan in \cite{Jiang1} and Liu and Zhang in \cite{Liu1}, see Liu and Hao \cite{Liu2} and Wang and Yu \cite{Wang} for the compressible case. The existence of weak solution to density dependent liquid crystals equations with term $|\nabla d|^2d$ for arbitrary initial data is not known in the present, see Jiang, Jiang and Wang \cite{Jiang2} for a result in this direction, where the global existence of weak solution to compressible liquid crystal equations in two dimensions is obtained under some geometric assumption on the initial direction, see also Wu and Tan \cite{WUTAN} for the global existence of low energy weak solutions in three dimensions. When the initial data gains more regularity, one can expect to obtain more regular solutions than the weak ones.
Wen and Ding \cite{WD} obtain the local existence and uniqueness of strong solutions to the Dirichlet problem of the system (\ref{1.1})--(\ref{1.5}) in bounded domain with initial
density being allowed to have vacuum. They also established the
global existence and uniqueness of solutions for two dimensional case
if the initial density is away from vacuum and the initial data is
of small norm. Global existence of strong solutions with small initial data to three dimensional liquid crystal equations are obtained by Li and Wang in \cite{LIXL1} for constant density case, Li and Wang in \cite{LIX2} for nonconstant but positive density case, and Ding, Huang and Xia in \cite{DHXIA} and Li \cite{Li} for nonnegative density case.

In the present paper, we consider the global existence of strong and weak solutions to the Cauchy problem of the system (\ref{1.1})--(\ref{1.5}). More precisely, if the initial data is regular and away from vacuum, we obtain the global strong solutions, and if vacuum appears initially, then we obtain the global weak solutions. As the first step of our procedure, recalling the local existence and uniqueness of strong solutions has been proven in \cite{WD} in bounded domain, we use the standard domain expanding method to obtain the local existence of strong solution to the Cauchy problem under the assumption that the initial density is away from zero. After obtaining the local strong solutions, the next step is to extend such local strong solution to be a global one. For this purpose, the main issue is to do the estimates on the local strong solutions which guarantees to extend the local solution to any finite time, obtaining the global strong solution. The energy inequalities established in Lemma \ref{lem1.2} shows that all the higher order estimates are based on space time $L^4$ bound of $\nabla d$. Using the rigidity theorem established recently by Lei, Li and Zhang in \cite{LEI}, we can successfully obtain the a priori bound on $L^4$ norm of $\nabla d$ in space and time, and thus finish the proof of global strong solutions. And finally, using the compactness results of Lions in \cite{Lions}, we can establish the global existence of weak solutions to the Cauchy problem.

Before stating our main results, we give the definitions of strong and weak solutions.

\begin{definition} Let $0<T<\infty$. $(\rho, u, d, p)$ is called a strong solution to the system (\ref{1.1})--(\ref{1.5}) in $Q_T=\mathbb R^2\times(0, T)$ with initial data $(\rho_0, u_0, d_0)$, if
\begin{eqnarray*}
&&\rho\in L^\infty(\mathbb R^2\times(0, T)),\quad\nabla\rho,\rho_t\in L^\infty(0, T; L^2(\mathbb R^2)),\\
&&\nabla p\in L^\infty(0, T; L^2(\mathbb R^2)\cap L^2(0, T; L^q(\mathbb R^2))\\
&&u\in L^\infty(0, T; H^2(\mathbb R^2))\cap L^2(0, T; W^{2,q}(\mathbb R^2)),\quad u_t\in L^\infty(0, T; L^2(\mathbb R^2))\cap L^2(0, T; H^1(\mathbb R^2)),\\
&&\nabla d\in L^\infty(0, T; H^2(\mathbb R^2)),\quad d_t\in L^\infty(0, T; H^1(\mathbb R^2))\cap L^2(0, T; H^2(\mathbb R^2))
\end{eqnarray*}
for $q\in[2, \infty)$, satisfies equations (\ref{1.1})--(\ref{1.5}) a.e. in $\mathbb R^2\times(0, T)$, and the initial data $(\rho_0, u_0, d_0)$ a.e. in $\mathbb R^2$.
\end{definition}

\begin{definition}
Let $0<T<\infty$. A triple $(\rho, u, d)$ is called a weak solution to the system (\ref{1.1})--(\ref{1.5}) $Q_T=\mathbb R^2\times(0, T)$ with initial data $(\rho_0, u_0, d_0)$, if
\begin{eqnarray*}
&&\rho\in L^\infty(\mathbb R^2\times(0, T)),\quad\sqrt\rho u\in L^\infty(0, T; L^2(\mathbb R^2)),\\
&&u\in L^2(0, T; H^1(\mathbb R^2)),\quad \textmd{div}u=0\mbox{ in }\mathcal D'(Q_T),\\
&&\nabla d\in L^\infty(0, T; L^2(\mathbb R^2)),\quad |d|=1\mbox{ a.e. on }Q_T,
\end{eqnarray*}
and the following hold true
\begin{eqnarray*}
&&\int_{\mathbb R^2}\rho_0\varphi(x, 0)dx+\int_0^T\int_{\mathbb R^2}(\rho\varphi_t+\rho u\nabla\varphi)dxdt=0,\\
&&\int_0^t\int_{\mathbb R^2}[(\nabla u-\rho u\otimes u-\nabla d\odot\nabla d):\nabla\psi-\rho u\psi_t] dxdt=\int_{\mathbb R^2}\rho_0u_0\psi(x,0)dx,\\
&&\int_0^T\int_{\mathbb R^2}[\nabla d:\nabla\phi-d\phi_t+((u\cdot\nabla )d-|\nabla d|^2d)\phi]dxdt=\int_{\mathbb R^2}d_0\phi(x, 0)dx
\end{eqnarray*}
for all $\varphi\in C_0^\infty(\mathbb R^2\times[0, T))$, $\psi\in C_0^\infty(\mathbb R^2\times[0, T))$ with $\textmd{div}\psi=0$ and $\phi\in C_0^\infty(\mathbb R^2\times[0, T)$.
\end{definition}

Our main results are stated in the following two theorems.

\begin{theorem}\label{thm1.1}
Assume that $0<\underline\rho\leq\rho_0(x)\leq\overline\rho<\infty$, $\nabla\rho_0\in L^2(\mathbb R^2)$, $u_0\in H^2(\mathbb R^2)$, $\nabla d_0\in H^2(\mathbb R^2)$ with $\textmd{div}u_0=0$, $|d_0|=1$ and $d_{0,3}\geq\varepsilon_0>0$, where $\underline\rho$, $\overline\rho$ and $\varepsilon_0$ are positive constants. Then there exists a unique global strong solution $(\rho, u, d)$ to the system (\ref{1.1})--(\ref{1.5}), complemented with the initial data $(\rho, u, d)|_{t=0}=(\rho_0, u_0, d_0)$.
\end{theorem}

\begin{theorem}\label{thm1.2}
Assume that $0\leq\rho_0(x)\leq\overline\rho<\infty$, $\rho_0-\tilde\rho\in L^q(\mathbb R^2)$ for some $q\in(1,\infty)$, $u_0\in L^2(\mathbb R^2)$, $\nabla d_0\in L^2(\mathbb R^2)$ with $\textmd{div}u_0=0$, $|d_0|=1$ and $d_{0,3}\geq\varepsilon_0>0$, where $\overline\rho$, $\tilde\rho
$ and $\varepsilon_0$ are positive constants, $d_{0,3}$ is the third component of the vector $d_0$. Then there exists a global weak solution $(\rho, u, d)$ to the system (\ref{1.1})--(\ref{1.5}), complemented with the initial data $(\rho, u, d)|_{t=0}=(\rho_0, u_0, d_0)$.
\end{theorem}

\begin{remark}
In theorem \ref{thm1.1} we assume that the initial density is away from vacuum, while in Theorem \ref{thm1.2}, the initial density is allowed to have vacuum.
\end{remark}

The rest of this paper is arranged as follows: In section \ref{sec1}, we do some estimates on the local strong solutions in the balls $B_R(0)$, including the a priori energy estimates and the existence time of strong solutions independent of the parameter $R$; In section \ref{sec3}, by taking the limit $R\rightarrow\infty$ of the solutions obtained in the previous section, we firstly establish the local existence and uniqueness of strong solutions to the Cauchy problem of liquid crystal equations via the standard domain expanding argument, and then extend such local strong solution to be a global one; and finally, in section \ref{sec4}, we obtain the global weak solution by using compactness argument.

\section{Estimates on the local strong solutions in bounded domains}\label{sec2}
\allowdisplaybreaks
In this section, as preparations of the next section, we do some estimates on the local strong solutions in the balls $B_R(0)$, including the a priori energy estimates and the existence time of strong solutions independent of the parameter $R$.

We first state the following local existence and uniqueness of strong solutions to the Dirichlet problem in bounded domain.

\begin{lemma}\label{lem1.1}(Local existence in bounded domain, see \cite{WD}) Assume that $\rho_0\geq0$, $\rho_0\in H^1(\Omega)\cap L^\infty(\Omega)$, $u_0\in H_0^1(\Omega)\cap H^2(\Omega)$, $d_0\in H^2(\Omega)$ with $\textmd{div}u_0=0$, $|d_0|=1$ on $\overline\Omega$, and the following compatible condition is valid
$$
\Delta u_0-\nabla p_0-\textmd{div}(\nabla d_0\odot\nabla d_0)=\sqrt{\rho_0}g_0
$$
for $(\rho_0, g_0)\in H^1(\Omega)\times L^2(\Omega)$. Then there exists a constant $T$, such that the system (\ref{1.1})--(\ref{1.5})
complemented with the initial and boundary conditions
\begin{eqnarray*}
&&(\rho, u, d)|_{t=0}=(\rho_0, u_0, d_0),\\
&&(u, d)|_{\partial\Omega}=(0, d_0),
\end{eqnarray*}
has a unique solution $(\rho, u, d, p)$ on $Q_T=\Omega\times(0, T)$, satisfying
\begin{eqnarray*}
&&\rho\in L^\infty(0, T; H^1(\Omega)\cap L^\infty(Q_T),\quad\rho_t\in L^\infty(0, T; L^2(\Omega)),\\
&&u\in L^\infty(0, T; H_0^1(\Omega)\cap H^2(\Omega)),\quad u_t\in L^2(0, T; H_0^1(\Omega)),\\
&&p\in L^\infty(0, T; H^1(\Omega)),\\
&&d\in L^\infty(0, T; H^3(\Omega)), \quad d_t\in L^\infty(0, T; H_0^1(\Omega)).
\end{eqnarray*}
\end{lemma}

Now, we state and prove our main result of this section, concerning the energy estimates on the local strong solutions and the existence time independent of the parameter $R$.

\begin{lemma}\label{lem1.2}(Estimates on the existence time and a priori estimates)
Let all the assumptions in Lemma \ref{lem1.1} hod true, and we assume in addition that $0<\underline\rho\leq\rho_0(x)\leq\overline\rho$ for positive constants $\underline\rho$ and $\overline\rho$. Let $(\rho, u, d, p)$ be the solution given in Lemma \ref{lem1.1}, with $\Omega=B_R(0)$, $R\geq1$. Set $e_0=C(\underline\rho, \overline\rho)(1+\|u_0\|_2^2+\|\nabla d_0\|_2^2)$,
$$
e_1=C(\overline\rho,\underline\rho)[1+\|u_0\|_4^4+\|\nabla u_0\|_2^2+\|\nabla d_0\|_{W^{1,4}}^4+\|\nabla d_0\|_6^6+\|\nabla^2d_0\|_{H^1}^2+(\|\nabla d_0\|_2^2+\|\nabla d_0\|_4^4)\|\nabla^2d_0\|_2^2],
$$
and
$$
e_2=C(\overline\rho, \underline\rho)(1+\|u_0\|_4^4+\|\nabla u_0\|_4^4+\|\nabla^2u_0\|_2^2+\|\nabla p_0\|_2^2+\|\nabla d_0\|_4^4+\|\nabla d_0\|_8^8+\|\nabla^2d_0\|_4^4+\|\nabla^3d_0\|_2^2),
$$
where $C(\overline\rho, \underline\rho)$ is a positive constant depending only on $\overline\rho$ and $\underline\rho$.

Then the existence time can be chosen depending only on $e=e_0+e_1+e_2$, and the following inequalities hold true
\begin{align*}
\sup_{0\leq s\leq t}&(\|u\|_2^2+\|\nabla d\|_2^2)+\int_0^t(\|\nabla u\|_2^2+\|\Delta d+|\nabla d|^2d\|_2^2)ds\leq e_0,\\
E_1(t)\leq& e_1(t+1)+e_0\int_0^t(\|\nabla u\|_2^2+\|\nabla d\|_4^4)(E_1(s)+e_1)ds,\\
E_2(t)\leq& e+e^5\int_0^t(1+E_1^5(s))(1+E_2(s))ds,\\
E(t)\leq& C(e)
\end{align*}
for any $0\leq t\leq T$, and
\begin{align*}
\sup_{0\leq t\leq T}&(\|\nabla\rho\|_2^2+\|\rho_t\|_2^2+\|\nabla p\|_2^2+\|u\|_{H^2}^2+\|\nabla d\|_{H^2}^2)\\
&+\int_0^T(\|\nabla^2u\|_q^2+\|u_t\|_{H^1}^2+\|d_t\|_{H^2}^2+\|d_{tt}\|_2^2)dt\leq C(e,\|\nabla\rho_0\|_2),
\end{align*}
where
\begin{eqnarray*}
&&E_1(t)=\sup_{0\leq s\leq t}(\|\nabla u\|_2^2+\|\nabla^2d\|_2^2)+\int_0^t(\|u_t\|_2^2+\|\nabla^2u\|_2^2+\|\nabla d_t\|_2^2+\|\nabla^3d\|_2^2)ds,\\
&&E_2(t)=\sup_{0\leq s\leq t}(\|u_t\|_2^2+\|\nabla d_t\|_2^2)+\int_0^t(\|\nabla u_t\|_2^2+\|d_{tt}\|_2^2+\|\nabla^2d_t\|_2^2)ds,
\end{eqnarray*}
and
$$
E(t)=E_1(t)+E_2(t)+e.
$$
\end{lemma}

\begin{proof}
Multiplying (\ref{1.2}) by $u$, (\ref{1.4}) by $-\Delta d$, summing the resulting equations up, integrating over $\Omega$, it follows from (\ref{1.1}), (\ref{1.3}) and (\ref{1.5}) that
$$
\frac{d}{dt}\left(\frac{\rho}{2}|u|^2+\frac{|\nabla d|^2}{2}\right)dx+\int_\Omega(|\nabla u|^2+|\Delta d+|\nabla d|^2d|^2)dx=0,
$$
and thus
\begin{equation}
\sup_{0\leq s\leq t}(\|\sqrt\rho u\|_2^2+\|\nabla d\|_2^2)+2\int_0^t(\|\nabla u\|_2^2+\|\Delta d+|\nabla d|^2d\|_2^2)ds\leq \|\sqrt{\rho_0}u_0\|_2^2+\|\nabla d_0\|_2^2\leq e_0.\label{1.6}
\end{equation}
Multiplying (\ref{1.2}) by $u_t$ and integrating over $\Omega$ yields
$$
\frac{d}{dt}\int_\Omega\frac{|\nabla u|^2}{2}dx+\int_\Omega\rho|u_t|^2dx=-\int_\Omega[\rho(u\cdot\nabla)u\cdot u_t+\Delta d\cdot\nabla d\cdot u_t]dx,
$$
from which, noticing that $0<\underline\rho\leq\rho(x, t)\leq\overline\rho$, we obtain, by Cauchy inequality, that
\begin{equation}\label{1.7}
\frac{d}{dt}\int_\Omega|\nabla u|^2dx+\underline\rho\int_\Omega\rho|u_t|^2dx\leq C(\overline\rho, \underline\rho)\int_\Omega(|u|^2|\nabla u|^2+|\nabla d|^2|\Delta d|^2)dx.
\end{equation}
Applying $H^2$ estimates to the Stokes equations, it follows from (\ref{1.2}) that
\begin{align*}
&\|\nabla^2u\|_2^2+\|\nabla p\|_2^2\leq C\|\rho(u_t+(u\cdot\nabla)u)-\Delta d\cdot\nabla d\|_2^2\\
\leq&C(\overline\rho)(\|u_t\|_2^2+\||u||\nabla u|\|_2^2+\||\Delta d||\nabla d|\|_2^2),
\end{align*}
which, combined with (\ref{1.7}), gives
\begin{align}
&\sup_{0\leq s\leq t}\|\nabla u\|_2^2+\int_0^t(\|u_t\|_2^2+\|\nabla^2u\|_2^2+\|\nabla p\|_2^2)ds\nonumber\\
\leq&C(\overline\rho, \underline\rho)\int_0^t\int_\Omega(|u|^2|\nabla u|^2+|\nabla d|^2|\Delta d|^2)dxds+C(\overline\rho, \underline\rho)\|\nabla u_0\|_2^2.\label{1.8}
\end{align}
Taking operator $\Delta$ on both sides of equation (\ref{1.4}) yields
\begin{align*}
&\Delta d_t+(u\cdot\nabla)\Delta d+2(\nabla u_i\cdot\partial_i\nabla)d+(\Delta u\cdot\nabla)d\\
=&\Delta^2d+|\nabla d|^2\Delta d+2\nabla|\nabla d|^2\cdot\nabla d+2(\nabla d\cdot\nabla \Delta d)d+2|\nabla^2d|^2d.
\end{align*}
Multiplying the above equation by $\Delta d$ and integrating over $\Omega$ yields
\begin{align}
\frac{d}{dt}\int_\Omega\frac{|\Delta d|^2}{2}dx-\int_\Omega\Delta^2d\Delta d dx\leq&\int_\Omega(2|\nabla u||\nabla^2d|^2+|\Delta u||\nabla d||\nabla^2d|\nonumber\\
&+7|\nabla d|^2|\nabla^2d|^2+2|\nabla d|^3|\nabla\Delta d|)dx,\label{1.14}
\end{align}
where we have used $-\Delta d\cdot d=|\nabla d|^2$ guaranteed by $|d|=1$. Integrating by parts deduce, noticing that $\Delta d|_{\partial\Omega}=-|\nabla d|^2d$ by equation (\ref{1.4}), that
\begin{align*}
-\int_\Omega\Delta^2d\Delta ddx=&\int_\Omega|\nabla\Delta d|^2dx-\int_{\partial\Omega}\frac{\partial\Delta d}{\partial n}\Delta d dS\\
=&\int_\Omega|\nabla\Delta d|^2dx+\int_{\partial\Omega}\frac{\partial\Delta d}{\partial n}|\nabla d|^2d dS\\
=&\int_\Omega|\nabla\Delta d|^2dx+\int_{\partial\Omega}(|\nabla d|^2\frac{\partial(\Delta d\cdot d)}{\partial n}-\Delta d|\nabla d|^2\frac{\partial d}{\partial n})dS\\
\geq&\int_\Omega|\nabla\Delta d|^2dx-3\int_{\partial\Omega}|\nabla d|^3|\nabla^2d|dS,
\end{align*}
which, combined with (\ref{1.14}), gives
\begin{align*}
&\frac{d}{dt}\int_\Omega\frac{|\Delta d|^2}{2}dx+\int_\Omega|\nabla\Delta d|^2dx\\
\leq&C\int_\Omega(|\nabla u||\nabla^2d|^2+|\nabla d|^2|\nabla^2d|^2+|\Delta u||\nabla d||\nabla^2d|+|\nabla d|^3|\nabla\Delta d|)dx\\
&+\int_{\partial\Omega}|\nabla d|^3|\nabla^2d|dS,
\end{align*}
and thus
\begin{align}
&\frac{d}{dt}\int_\Omega\frac{|\Delta d|^2}{2}dx+\int_\Omega|\nabla\Delta d|^2dx\nonumber\\
\leq&\varepsilon\int_\Omega|\Delta u|^2dx+C\int_\Omega(|\nabla u||\nabla^2d|^2+|\nabla d|^2|\nabla^2d|^2+|\nabla d|^6)dx
+\int_{\partial\Omega}|\nabla d|^3|\nabla^2d|dS.\label{1.15-1}
\end{align}
We compute
\begin{align*}
&\int_{\partial\Omega}|\nabla d|^3|\nabla^2d|dS=\int_{\partial B_R(0)}|\nabla d|^3|\nabla^2d|\frac{x}{R}\cdot ndS
=\int_{B_R(0)}\textmd{div}\left(\frac{x}{R}|\nabla d|^3|\nabla^2d|\right)dx\\
\leq&\frac{2}{R}\int_{B_R(0)}|\nabla d|^3|\nabla^2d|dx+C\int_{B_R(0)}(|\nabla d|^2|\nabla^2d|^2+|\nabla d|^3|\nabla^3d|)dx\\
\leq&C\int_{B_R(0)}(|\nabla d|^3|\nabla^2d|+|\nabla d|^2|\nabla^2d|^2+|\nabla d|^6)dx+\varepsilon\int_{B_R(0)}|\nabla^3d|^2dx,
\end{align*}
which, combined with (\ref{1.15-1}), gives
\begin{align}
&\frac{d}{dt}\int_\Omega|\Delta d|^2dx+\int_\Omega|\nabla\Delta d|^2dx\nonumber\\
\leq&\varepsilon\int_\Omega(|\Delta u|^2+|\nabla^3d|^2)dx+C\int_\Omega(|\nabla d|^2|\nabla^2d|^2+|\nabla u||\nabla^2d|^2+|\nabla d|^3|\nabla^2d|+|\nabla d|^6)dx.\label{1.16-0}
\end{align}
Elliptic estimates give
\begin{align}
&\int_\Omega|\nabla^3d|^2dx\leq\|\nabla^3(d-d_0)\|_2^2+\|\nabla^3d_0\|_2^2\nonumber\\
\leq& C\|\nabla\Delta(d-d_0)\|_{H^1}^2+\|\nabla^3d_0\|_2^2\leq C(\|\Delta d\|_{H^1}^2+\|\nabla^2d_0\|_{H^1}^2).\label{1.16-1}
\end{align}
Combining (\ref{1.16-0}) with (\ref{1.16-1}), together with (\ref{1.8}), there holds
\begin{align}
&\sup_{0\leq s\leq t}(\|\nabla u\|_2^2+\|\nabla^2d\|_2^2)+\int_0^t(\|u_t\|_2^2+\|\nabla^2u\|_2^2+\|\nabla^3d\|_2^2)ds\nonumber\\
\leq&C(\overline\rho, \underline\rho)(\|\nabla u_0\|_2^2+\|\nabla^2d_0\|_2^2+\|\nabla^2d_0\|_{H^1}^2t)+C(\overline\rho, \underline\rho)\int_0^t\int_\Omega(|u|^2|\nabla u|^2\nonumber\\
&+|\nabla d|^2|\nabla^2d|^2+|\nabla u||\nabla^2d|^2+|\nabla^2d|^2+|\nabla d|^6)dx. \label{1.16-2}
\end{align}
By Ladyzhenskaya inequality, Gagliado-Nirenberg inequality, H\"older inequality and Cauchy inequality, we can estimate the terms on the right hand side of (\ref{1.6-2}) as follows
\begin{align*}
I_1=&\int_0^t\int_\Omega|u|^2|\nabla u|^2dxds\leq\int_0^t\|u\|_4^2\|\nabla u\|_4^2ds\\
\leq&C\int_0^t\|u\|_2\|\nabla u\|_2\|\nabla u\|_2\|\nabla^2u\|_2ds\leq\varepsilon\int_0^t\|\nabla^2u\|_2^2ds+C\int_0^t\|u\|_2^2\|\nabla u\|_2^4ds,\\
I_2=&\int_0^t\int_\Omega(|\nabla d|^2+|\nabla u|)|\nabla^2d|^2dxds\leq\int_0^t(\|\nabla d\|_4^2+\|\nabla u\|_2)\|\nabla^2d\|_4^2ds\\
\leq&C\int_0^t(\|\nabla d\|_4^2+\|\nabla u\|_2)(\|\nabla^2(d-d_0)\|_4^2+\|\nabla^2d_0\|_4^2)ds\\
\leq&C\int_0^t(\|\nabla d\|_4^2+\|\nabla u\|_2)(\|\nabla^2(d-d_0)\|_2\|\nabla^3(d-d_0)\|_2+\|\nabla^2d_0\|_4^2)ds\\
\leq&C\int_0^t(\|\nabla d\|_4^2+\|\nabla u\|_2)(\|\nabla^2d_0\|_4^2+\|\nabla^2d_0\|_2\|\nabla^3d_0\|_2+\|\nabla^2d_0\|_2\|\nabla^3d\|_2\\
&+\|\nabla^3d_0\|_2\|\nabla^2d\|_2+\|\nabla^2d\|_2\|\nabla^3d\|_2+\|\nabla^2d_0\|_4^2)ds\\
\leq&\varepsilon\int_0^t\|\nabla^3d\|_2^2ds+C\int_0^t(\|\nabla d\|_4^4+\|\nabla u\|_2^2)(\|\nabla^2d\|_2^2+\|\nabla^2 d_0\|_2^2+1)ds\\
&+Ct(\|\nabla^2d_0\|_2^2+\|\nabla^3d_0\|_2^2+\|\nabla^2d_0\|_4^4),\\
I_3=&\int_0^t\int_\Omega|\nabla d|^6dxds\leq C\int_0^t\int_\Omega(|\nabla(d-d_0)|^6+|\nabla d_0|^6)dxds\\
\leq&C\int_0^t(\|\nabla(d-d_0)\|_4^4\|\nabla^2(d-d_0)\|_2^2+\|\nabla d_0\|_6^6)ds\\
\leq&C\int_0^t(\|\nabla d\|_4^4\|\nabla^2d\|_2^2+\|\nabla^2d_0\|_2^2\|\nabla d\|_4^4+\|\nabla d_0\|_4^4\|\nabla^2d\|_2^2\\
&+\|\nabla d_0\|_4^4\|\nabla^2d_0\|_2^2+\|\nabla d_0\|_6^6)ds.
\end{align*}
Substituting these inequalities into (\ref{1.16-2}) yields
\begin{align}
&\sup_{0\leq s\leq t}(\|\nabla u\|_2^2+\|\nabla^2d\|_2^2)+\int_0^t(\|\nabla^2u\|_2^2+\|u_t\|_2^2+\|\nabla^3d\|_2^2)ds\nonumber\\
\leq&C(\overline\rho,\underline\rho)[\|\nabla u_0\|_2^2+\|\nabla^2d_0\|_2^2+(\|\nabla^2d_0\|_{H^1}^2+\|\nabla^2d_0\|_4^4+\|\nabla d_0\|_6^6+\|\nabla d_0\|_4^4\|\nabla^2d_0\|_2^2)t]\nonumber\\
&+C(\overline\rho,\underline\rho)\int_0^t[\|u\|_2^2\|\nabla u\|_2^4+(\|\nabla d\|_4^4+\|\nabla u\|_2^2)(\|\nabla^2d\|_2^2+\|\nabla^2d_0\|_2^2+1)\nonumber\\
&+\|\nabla d_0\|_4^4\|\nabla^2d\|_2^2]ds.\label{1.16-3-1}
\end{align}
Set
$$
e_1=C(\overline\rho,\underline\rho)[1+\|u_0\|_4^4+\|\nabla u_0\|_2^2+\|\nabla d_0\|_{W^{1,4}}^4+\|\nabla d_0\|_6^6+\|\nabla^2d_0\|_{H^1}^2+(\|\nabla d_0\|_2^2+\|\nabla d_0\|_4^4)\|\nabla^2d_0\|_2^2],
$$
and
$$
E_1(t)=\sup_{0\leq s\leq t}(\|\nabla u\|_2^2+\|d_t\|_2^2+\|\nabla^2d\|_2^2)+\int_0^t(\|u_t\|_2^2+\|\nabla^2u\|_2^2+\|\nabla d_t\|_2^2+\|\nabla^3d\|_2^2)ds,
$$
It follows from (\ref{1.6}) and (\ref{1.16-3-1}) that
\begin{equation}\label{1.16-3}
E_1(t)\leq e_1(t+1)+e_0\int_0^t(\|\nabla u\|_2^2+\|\nabla d\|_4^4)(E_1(s)+e_1)ds.
\end{equation}
By Ladyzhenskaya inequality, it follows from elliptic estimates that
\begin{align}
\|\nabla d\|_4^4\leq& C(\|\nabla(d-d_0)\|_4^4+\|\nabla d_0\|_4^4)\leq C(\|\nabla(d-d_0)\|_2^2\|\nabla^2(d-d_0)\|_2^2+\|\nabla d_0\|_4^4)\nonumber\\
\leq&C(\|\nabla d\|_2^2\|\nabla^2d\|_2^2+\|\nabla d_0\|_2^2\|\nabla^2d\|_2^2+\|\nabla^2d_0\|_2^2\|\nabla d\|_2^2+\|\nabla d_0\|_2\|\nabla^2d_0\|_2+\|\nabla d_0\|_4^4)\nonumber\\
\leq&C(e_0+e_1+e_0\|\nabla^2d\|_2^2),\label{1.16-5}
\end{align}
which, substituted into (\ref{1.16-3}), gives
\begin{align*}
E_1(t)\leq&e_1(1+t)+e_0\int_0^t(E_1(s)+e_0+e_1+e_0E_1(s))(E_1(s)+e_1)ds\\
\leq&e_1(1+t)+e_0(e_0+e_1)\int_0^t(E_1(s)+e_1)^2ds,
\end{align*}
and thus
$$
E_1(t)+e_1\leq e_1(1+t)+e_0(e_0+e_1)\int_0^t(E_1(s)+e_1)^2ds.
$$
Noticing that $E_1(t)$ is increasing, the above inequality implies
\begin{equation}\label{1.16-6}
E_1(t)+e_1\leq e_1(1+t)+e_0(e_0+e_1)(E_1(t)+e_1)^2t.
\end{equation}

Differentiate (\ref{1.2}) with respect to $t$ and using (\ref{1.1}), it has
\begin{align*}
\rho(u_{tt}+(u\cdot\nabla)u_t)-\Delta u_t+\nabla p_t=&\textmd{div}(\rho u)(u_t+(u\cdot\nabla)u)-\rho(u_t\cdot\nabla)u\\
&-\textmd{div}(\nabla d_t\odot\nabla d-\textmd{div}(\nabla d\odot\nabla d_t).
\end{align*}
Multiplying the above equation by $u_t$, and integrating over $\Omega$ yields
\begin{align*}
&\frac{d}{dt}\int_\Omega\frac{\rho}{2}|u_t|^2dx+\int_\Omega|\nabla u_t|^2dx\\
\leq&C\int_\Omega(|\nabla u_t||\nabla d||\nabla d_t|+\rho|u||u_t||\nabla u_t|+\rho|u|^2|\nabla u||\nabla u_t|+\rho|u|^2|\nabla^2u||u_t|\\
&+\rho|u||\nabla u|^2|u_t|+\rho|\nabla u||u_t|^2)dx\\
\leq&\frac{1}{2}\int_\Omega|\nabla u_t|^2dx+C\int_\Omega(|\nabla d|^2|\nabla d_t|^2+\rho^2|u|^2|u_t|^2+\rho^2|u|^4|\nabla u|^2\\
&+\rho|u|^2|\nabla^2u||u_t|+\rho|u||\nabla u|^2|u_t|+\rho|\nabla u||u_t|^2)dx,
\end{align*}
and thus
\begin{align}
&\sup_{0\leq s\leq t}\|u_t\|_2^2+\int_0^t\|\nabla u_t\|_2^2ds\nonumber\\
\leq&\|\sqrt{\rho_0}u_t(0)\|_2^2+C(\overline\rho)\int_0^t\int_\Omega(|\nabla d|^2|\nabla d_t|^2+\rho^2|u|^2|u_t|^2+\rho^2|u|^4|\nabla u|^2\nonumber\\
&+\rho|u|^2|\nabla^2u||u_t|+\rho|u||\nabla u|^2|u_t|+\rho|\nabla u||u_t|^2)dx.\label{1.15}
\end{align}
By Ladyzhenskaya inequality, Gagliado-Nirenberg inequality and Cauchy inequality, we can estimate the terms on the right hand side of the above inequality as follows
\begin{align*}
J_1=&\int_0^t\int_\Omega|u||\nabla u|^2|u_t|dxds\leq\int_0^t\|u\|_4\|\nabla u\|_4\|u_t\|_4ds\\
\leq&C\int_0^t\|u\|_2^{1/2}\|\nabla u\|_2^{1/2}\|\nabla u\|_2^{1/2}\|\nabla^2u\|_2^{1/2}\|u_t\|_2^{1/2}\|\nabla u_t\|_2^{1/2}ds\\
\leq&C\int_0^t\|u\|_2^{1/2}\|\nabla u\|_2\|\nabla^2 u\|_2^{1/2}\|u_t\|_2^{1/2}\|\nabla u_t\|_2^{1/2}ds\\
\leq&\varepsilon\int_0^t(\|\nabla u_t\|_2^2+\|\nabla^2u\|_2^2)ds+C\int_0^t\|u\|_2\|\nabla u\|_2^2\|u_t\|_2ds,\\
J_2=&\int_0^t\int_\Omega|u|^2|u_t|^2dxds\leq\int_0^t\|u\|_4^2\|u_t\|_4^2ds\leq C\int_0^t\|u\|_2\|\nabla u\|_2\|u_t\|_2\|\nabla u_t\|_2ds\\
\leq&\varepsilon\int_0^t\|\nabla u_t\|_2^2dx+C\int_0^t\|u\|_2^2\|\nabla u\|_2^2\|u_t\|_2^2ds,\\
J_3=&\int_0^t\int_\Omega|u|^2|\nabla^2u||u_t|dxds\leq\int_0^t\|u\|_8^2\|\nabla^2u\|_2\|u_t\|_4ds\\
\leq&C\int_0^t\|u\|_2^{1/2}\|\nabla u\|_2^{3/2}\|\nabla^2u\|_2\|u_t\|_2^{1/2}\|\nabla u_t\|_2^{1/2}ds\\
\leq&\varepsilon\int_0^t(\|\nabla u_t\|_2^2+\|\nabla^2u\|_2^2)ds+C\int_0^t\|u\|_2^2\|\nabla u\|_2^6\|u_t\|_2^2ds,\\
J_4=&\int_0^t\int_\Omega|u|^4|\nabla u|^2dxds\leq\int_0^t\|u\|_8^4\|\nabla u\|_4^2ds\\
\leq&C\int_0^t\|u\|_2\|\nabla u\|_2^3\|\nabla u\|_2\|\nabla^2u\|_2ds\leq\varepsilon\int_0^t\|\nabla^2u\|_2^2ds+C\int_0^t\|u\|_2^2\|\nabla u\|_2^8ds,\\
J_5=&\int_0^t\int_\Omega|\nabla u||u_t|^2dxde\leq\int_0^t\|\nabla u\|_2\|u_t\|_4^2ds\\
\leq&C\int_0^t\|\nabla u\|_2\|u_t\|_2\|\nabla u_t\|_2ds\leq\varepsilon\int_0^t\|\nabla u_t\|_2^2ds+C\int_0^t\|\nabla u\|_2^2\|u_t\|_2^2ds.
\end{align*}
Substituting the above inequalities into (\ref{1.15}), it follows that
\begin{align}
&\sup_{0\leq s\leq t}\|u_t\|_2^2+\int_0^t\|\nabla u_t\|_2^2ds\nonumber\\
\leq&C\|\sqrt{\rho_0}u_t(0)\|_2^2+C(\overline\rho)\int_0^t(\|u\|_2\|\nabla u\|_2^2\|u_t\|_2+\|u\|_2^2\|\nabla u\|_2^2\|u_t\|_2^2+\|u\|_2^2\|\nabla u\|_2^6\|u_t\|_2^2\nonumber\\
&+\|u\|_2^2\|\nabla u\|_2^8+\|\nabla u\|_2^2\|u_t\|_2^2)ds+C(\overline\rho)\int_0^t\int_\Omega|\nabla d|^2|\nabla d_t|^2dxds.\label{1.16}
\end{align}
Differentiate equation (\ref{1.4}) with respect to $t$, there holds
$$
d_{tt}-\Delta d_t=|\nabla d|^2d_t+2\nabla d:\nabla d_t d-(u\cdot\nabla)d_t-(u_t\cdot\nabla )d.
$$
Taking square power to both sides of the above equation and integrating over $\Omega$, we get
\begin{align*}
&\frac{d}{dt}\int_\Omega|\nabla d_t|^2dx+\int_\Omega|\Delta d_t|^2dx\\
\leq&C\int_\Omega(|\nabla d|^4|d_t|^2+|\nabla d|^2|\nabla d_t|^2+|u|^2|\nabla d_t|^2+|u_t|^2|\nabla d|^2)dx,
\end{align*}
and thus
\begin{align}
&\sup_{0\leq s\leq t}\|\nabla d_t\|_2^2+\int_0^t\|\Delta d_t\|_2^2ds\nonumber\\
\leq&\|\nabla d_t(0)\|_2^2+C\int_0^t\int_\Omega(|\nabla d|^4|d_t|^2+|\nabla d|^2|\nabla d_t|^2+|u|^2|\nabla d_t|^2+|u_t|^2|\nabla d|^2)dxds.\label{1.17}
\end{align}
By Ladyzhenskaya inequality, Gagliado-Nirenberg inequality and Cauchy inequality, we have the following estimates
\begin{align*}
K_1=&\int_0^t\int_\Omega|\nabla d|^4|d_t|^2dxds\leq\int_0^t\|\nabla d\|_8^4\|d_t\|_4^2ds\\
\leq& C\int_0^t(\|\nabla(d-d_0)\|_8^4+\|\nabla d_0\|_8^4)\|d_t\|_2\|\nabla d_t\|_2ds\nonumber\\
\leq&C\int_0^t(\|\nabla(d-d_0)\|_2\|\nabla^2(d-d_0)\|_2^3+\|\nabla d_0\|_8^4)\|d_t\|_2\|\nabla d_t\|_2ds\nonumber\\
\leq&C\int_0^t[(\|\nabla d\|_2+\|\nabla d_0\|_2)(\|\nabla^2d\|_2^3+\|\nabla^2d_0\|_2^3)+\|\nabla d_0\|_8^4]\|d_t\|_2\|\nabla d_t\|_2ds\nonumber\\
\leq&C\int_0^t(\|\nabla d_0\|_2^4+\|\nabla^2d_0\|_2^4+\|\nabla d_0\|_8^4+\|\nabla d\|_2^4+\|\nabla^2d\|_2^4)\|d_t\|_2\|\nabla d_t\|_2ds,\\
K_2=&\int_0^t\int_\Omega|\nabla d|^2|\nabla d_t|^2dxds\leq\int_0^t\int_\Omega\|\nabla d\|_4^2\|\nabla d_t\|_4^2ds\nonumber\\
\leq&C\int_0^t\|\nabla d\|_4^2\|\nabla d_t\|_2\|\Delta d_t\|_2ds\leq\varepsilon\int_0^t\|\Delta d_t\|_2^2ds+C\int_0^t\|\nabla d\|_4^4\|\nabla d_t\|_2^2ds,\\
K_3=&\int_0^t\int_\Omega|u|^2|\nabla d_t|^2dxds\leq\int_0^t\|u\|_4^2\|\nabla d_t\|_4^2ds\leq C\int_0^t\|u\|_2\|\nabla u\|_2\|\nabla d_t\|_2\|\Delta d_t\|_2ds\nonumber\\
\leq&\varepsilon\int_0^t\|\Delta d_t\|_2^2ds+C\int_0^t\|u\|_2^2\|\nabla u\|_2^2\|\nabla d_t\|_2^2ds,\\
K_4=&\int_0^t\int_\Omega|\nabla d|^2|u_t|^2dxds\leq\int_0^t\|\nabla d\|_4^2\|u_t\|_4^2ds\nonumber\\
\leq&C\int_0^t\|\nabla d\|_4^2\|u_t\|_2\|\nabla u_t\|_2ds\leq\varepsilon\int_0^t\|\nabla u_t\|_2^2ds+C\int_0^t\|\nabla d\|_4^4\|u_t\|_2^2ds.\nonumber
\end{align*}
Substituting these inequalities into (\ref{1.17}) yields
\begin{align}
&\sup_{0\leq s\leq t}\|\nabla d_t\|_2^2+\int_0^t\|\nabla^2d_t\|_2^2ds\nonumber\\
\leq&C\|\nabla d_t(0)\|_2^2+C(\|\nabla d_0\|_2^4+\|\nabla^2d_0\|_2^4+\|\nabla d_0\|_8^4+1)\int_0^t[(\|\nabla d\|_2^4+\|\nabla^2d\|_2^4)\nonumber\\
&\times\|d_t\|_2\|\nabla d_t\|_2+\|\nabla d\|_4^4(\|u_t\|_2^2+\|\nabla d_t\|_2^2)\nonumber\\
&+\|u\|_2^2\|\nabla u\|_2^2\|\nabla d_t\|_2^2]ds+\varepsilon\int_0^t\|\nabla u_t\|_2^2ds.\label{1.19}
\end{align}
Combining (\ref{1.16}) with the inequality of $K_2$, together with (\ref{1.19}), we get
\begin{align}
&\sup_{0\leq s\leq t}(\|u_t\|_2^2+\|\nabla d_t\|_2^2)+\int_0^t(\|\nabla u_t\|_2^2+\|\nabla^2d_t\|_2^2)ds\nonumber\\
\leq&C(\|u_t(0\|_2^2+\|\nabla d_t(0)\|_2^2)+C(\overline\rho)(\|\nabla d_0\|_2^4+\|\nabla^2d_0\|_2^4+\|\nabla d_0\|_8^4+1)\nonumber\\
&\times\int_0^t[(1+\|u\|_2^2)\|\nabla u\|_2^2\|u_t\|_2^2+\|\nabla u\|_2^2+\|u\|_2^2\|\nabla u\|_2^6\|u_t\|_2^2+\|u\|_2^2\|\nabla u\|_2^8\nonumber\\
&+(\|\nabla d\|_2^4+\|\nabla^2d\|_2^4)\|d_t\|_2\|\nabla d_t\|_2+\|\nabla d\|_4^4(\|u_t\|_2^2+\|\nabla d_t\|_2^2)\nonumber\\
&+\|u\|_2^2\|\nabla u\|_2^2\|\nabla d_t\|_2^2]ds.\label{1.20}
\end{align}
Using equations (\ref{1.2}) and (\ref{1.4}), we compute
\begin{align*}
\|u_t(0)\|_2^2\leq& C(\underline\rho)\|(\rho u_t)(0)\|_2^2\leq C(\underline\rho)\|-\Delta d_0\cdot\nabla d_0-\nabla p_0+\Delta u_0-\rho_0(u_0\cdot\nabla )u_0\|_2^2\\
\leq&C(\underline\rho, \overline\rho)(\|\Delta d_0\|_4^2\|\nabla d_0\|_4^2+\|\nabla p_0\|_2^2+\|\nabla^2u_0\|_2^2+\|u_0\|_4^2\|\nabla u_0\|_4^2)\\
\leq&C(\overline\rho, \underline\rho)(\|\nabla d_0\|_4^4+\|\nabla^2d_0\|_4^4+\|\nabla p_0\|_2^2+\|\nabla^2u_0\|_2^2+\|u_0\|_4^4+\|\nabla u_0\|_4^4)
\end{align*}
and
\begin{align*}
\|\nabla d_t(0)\|_2^2\leq&C(\|\nabla^3d_0\|_2^2+\|\nabla d_0\|_6^6+\|\nabla d_0\|_4^2\|\nabla^2d_0\|_4^2+\|u_0\|_4^2\|\nabla^2d_0\|_4^2+\|\nabla u_0\|_4^2\|\nabla d_0\|_4^2\\
\leq&C(\|\nabla^3d_0\|_2^2+\|\nabla d_0\|_6^6+\|\nabla d_0\|_4^4+\|\nabla^2d_0\|_4^4+\|u_0\|_4^4+\|\nabla u_0\|_4^4).
\end{align*}
Set
\begin{align*}
e_2=&C(\overline\rho, \underline\rho)(1+\|u_0\|_4^4+\|\nabla u_0\|_4^4+\|\nabla^2u_0\|_2^2+\|\nabla p_0\|_2^2\\
&+\|\nabla d_0\|_4^4+\|\nabla d_0\|_8^8+\|\nabla^2d_0\|_4^4+\|\nabla^3d_0\|_2^2),
\end{align*}
and
$$
E_2(t)=\sup_{0\leq s\leq t}(\|u_t\|_2^2+\|\nabla d_t\|_2^2)+\int_0^t(\|\nabla u_t\|_2^2+\|\nabla^2d_t\|_2^2)ds.
$$
Then, it follows from (\ref{1.6}), (\ref{1.16-5}) and (\ref{1.20}) that
\begin{align*}
E_2(t)\leq&e_2+e_2\int_0^t[e_0E_1(s)E_2(s)+e_0E_1^3(s)E_2(s)+e_0E_1^4(s)+(e_0^2+E_1^2(s))E_1(s)^{1/2}E_2^{1/2}(s)\nonumber\\
&+(e_0+e_1)(1+E_1(s))E_2(s)+e_0E_1(s)E_2(s)]ds\nonumber\\
\leq&e_2+e_2\int_0^t[(e_0E_1(s)+e_0E_1^3(s)+1+e_0+e_1+(e_0+e_1)E_1(s)+e_0E_1(s))E_2(s)\\
&+e_0E_1^4(s)+(e_0^2+E_1^2(s))^2E_1(s)]ds,\\
\leq&e_2+e_2\int_0^t[(e_0+e_1)(1+E_1^3(s))E_2(s)+(e_0+e_1)^4(1+E_1^5(s))]ds\\
\leq&e_2+(e_0+e_1+e_2)^5\int_0^t(1+E_1^5(s))(1+E_2(s))ds.
\end{align*}
Noticing that $E_2(t)$ is nondecreasing, it follows from the above inequality that
\begin{equation}\label{1.21}
E_2(t)\leq e_2+(e_0+e_1+e_2)^5(1+E_1^5(t))(1+E_2(t))t.
\end{equation}
Set $e=e_0+e_1+e_2$, and $E(t)=E_1(t)+E_2(t)+e$, then it follows from (\ref{1.16-6}) and (\ref{1.21}) that
$$
E(t)\leq e(1+t)+e^5 E^6(t)t,
$$
from which, using continuity argument, one can easily show that
\begin{equation}\label{1.22}
E(t)\leq C(e),\quad\mbox{ for }t\in(0, t_e),
\end{equation}
for some positive constant $t_e$ depending only on $e$.

By Ladyzhenskaya inequality and Gagliado-Nirenberg inequality, and applying elliptic estimates to Stokes equations, we obtain
\allowdisplaybreaks\begin{align}
&\|\nabla^2u\|_2^2+\|\nabla p\|_2^2\leq C\|\rho(u_t+(u\cdot\nabla)u)\|_2^2+\|\nabla^2 d|\nabla d|\|_2^2\nonumber\\
\leq&C(\overline\rho)(\|u_t\|_2^2+\||u||\nabla u|\|_2^2+\||\nabla^2 d||\nabla d|\|_2^2)\nonumber\\
\leq&C(\overline\rho)(\|u_t\|_2^2+\|u\|_4^2\|\nabla u\|_4^2+\|\nabla^2 d\|_4^4\|\nabla d\|_4^2)\nonumber\\
\leq&C(\overline\rho)(\|u_t\|_2^2+\|u\|_2\|\nabla u\|_2^2\|\nabla^2u\|_2+\|\nabla^2 d\|_4^2\|\nabla d\|_4^2)\nonumber\\
\leq&\varepsilon\|\nabla^2u\|_2^2+\varepsilon(\|\nabla^2(d-d_0)\|_4^4+\|\nabla^2 d_0\|_4^4)+C(\overline\rho)(\|u_t\|_2^2+\|u\|_2^2\|\nabla u\|_2^4\nonumber\\
&+\|\nabla(d-d_0)\|_4^4+\|\nabla d_0\|_4^4)\nonumber\\
\leq&\varepsilon\|\nabla^2u\|_2^2+\varepsilon(C\|\nabla^2(d-d_0)\|_2^2\|\nabla^3(d-d_0)\|_2^2+\|\nabla^2 d_0\|_4^4)\nonumber\\
&+C(\overline\rho)(\|u_t\|_2^2+\|u\|_2^2\|\nabla u\|_2^4+\|\nabla(d-d_0)\|_2^2\|\nabla^2(d-d_0)\|_2^2+\|\nabla d_0\|_4^4)\nonumber\\
\leq&\varepsilon\|\nabla^2u\|_2^2+\varepsilon[C(e)(\|\nabla^3d\|_2^2+1)+C(\overline\rho)(\|u_t\|_2^2+C(e))]\nonumber\\
\leq&\varepsilon C(e)(\|\nabla^2u\|_2^2+\|\nabla^3d\|_2^2+\|u_t\|_2^2)+C(e), \label{1.23}
\end{align}
and
\begin{align}
&\|\nabla^2u\|_q^2+\|\nabla p\|_q^2\leq C\|\rho(u_t+(u\cdot\nabla)u)\|_q^2+\|\nabla^2 d|\nabla d|\|_q^2\nonumber\\
\leq&C(\|u_t\|_q^2+\||u||\nabla u|\|_q^2+\||\nabla^2 d||\nabla d|\|_q^2)\nonumber\\
\leq&C(\|u_t\|_2^{2(1-\theta)}\|\nabla u_t\|_2^{2\theta}+\|u\|_{2q}^2\|\nabla u\|_{2q}^2+\|\nabla^2 d\|_{2q}^2\|\nabla d\|_{2q}^2)\nonumber\\
\leq&C[\|u_t\|_2^{2(1-\theta)}\|\nabla u_t\|_2^{2\theta}+\|u\|_2^{2(1-\delta)}\|\nabla u\|_2^2\|\nabla^2u\|_2^{2\delta}\nonumber\\
&+(\|\nabla(d-d_0)\|_2^{2(1-\delta)}\|\nabla^2(d-d_0)\|_2^{2\delta}+\|\nabla d_0\|_{2q}^2)\nonumber\\
&\times(\|\nabla^2(d-d_0)\|_2^{2(1-\delta)}\|\nabla^3(d-d_0)\|_2^{2\delta}+\|\nabla^2 d_0\|_{2q}^2)]\nonumber\\
\leq&C(\|u_t\|_2^{2(1-\theta)}\|\nabla u_t\|_2^{2\theta}+\|u\|_2^{2(1-\delta)}\|\nabla u\|_2^2\|\nabla^2u\|_2^{2\delta}+C(e)\|\nabla^3d\|_2^{2\delta}+C(e))\nonumber\\
\leq&C(e)(\|\nabla u_t\|_2^{2\theta}+\|\nabla^3 d\|_2^{2\delta}+1)\leq\varepsilon(\|\nabla u_t\|_2^2+\|\nabla^3 d\|_2^2)+C(e)\label{1.24}
\end{align}
for any $2\leq q<\infty$, where $\theta=1-2/q$ and $\delta=1-1/q$. Using Ladyzhenskaya inequality and applying
elliptic estimates to Laplace equation yields
\begin{align}
&\|\nabla^3d\|_2^2\leq C(\|\nabla^3(d-d_0)\|_2^2+\|\nabla^3d_0\|_2^2)\leq C(\|\Delta(d-d_0)\|_{H^1}^2+\|\nabla^3d_0\|_2^2)\nonumber\\
\leq&C(\|d_t+(u\cdot\nabla)d-|\nabla d|^2d\|_{H^1}^2+\|\nabla^3d_0\|_2^2)\nonumber\\
\leq&C(\|\nabla d_t\|_2^2+\||\nabla u||\nabla d|\|_2^2+\||u||\nabla^2d|\|_2^2+\||\nabla d|^2\|_2^2+\||\nabla d||\nabla^2d|\|_2^2\nonumber\\
&+\|d_t\|_2^2+\||u|\nabla d\|_2^2+\||\nabla d|\|_2^2+\|\nabla^3d_0\|_2^2) \nonumber\\
\leq&C[\|\nabla d_t\|_2^2+\|\nabla u\|_2^2\|\nabla^2u\|_2^2(\|\nabla(d-d_0)\|_2\|\nabla^2(d-d_0)\|_2^2+\|\nabla d_0\|_4^4)\nonumber\\
&+\|u\|_2\|\nabla u\|_2(\|\nabla^2(d-d_0)\|_2\|\nabla^3(d-d_0)\|_2+\|\nabla^2d_0\|_4^4)\nonumber\\
&+\|\nabla(d-d_0)\|_2^2\|\nabla^2(d-d_0)\|_2^4+\|\nabla d_0\|_6^6+(\|\nabla(d-d_0)\|_2\|\nabla^2(d-d_0)\|_2\nonumber\\
&+\|\nabla d_0\|_4^2)(\|\nabla^2(d-d_0)\|_2\|\nabla^3(d-d_0)\|_2+\|\nabla^2d_0\|_4^2)\nonumber\\
&+\|d_t\|_2^2+\|u\|_2\|\nabla u\|_2(\|\nabla(d-d_0)\|_2+\|\nabla d_0\|_2)\nonumber\\
&\times(\|\nabla^2(d-d_0)\|_2+\|\nabla^2d_0\|_2)+\|\nabla(d-d_0)\|_2\nonumber\\
&+\|\nabla(d-d_0)\|_2\|\nabla^2(d-d_0)\|_2+\|\nabla^3d_0\|_2^2]\leq C(\|\nabla d_t\|_2^2+C(e)).\label{1.25}
\end{align}
Combining (\ref{1.22})--(\ref{1.25}), by taking $\varepsilon$ small enough, there holds
\begin{align*}
&\sup_{0\leq t\leq t_e}(\|u\|_{H^2}^2+\|\nabla d\|_{H^2}^2+\|\nabla p\|_2^2)+\int_0^{t_e}(\|\nabla^3u\|_q^2+\|\nabla p\|_q^2+\|\nabla d\|_{H^2}^2\\
&+\|d_t\|_{H^2}^2+\|d_{tt}\|_2^2)dt\leq C(e),\qquad q\in[2,\infty).
\end{align*}
Taking the operator $\nabla$ on both sides of equation (\ref{1.1}), multiplying the resulting equation by $2\nabla\rho$ and integrating over $\Omega$ yields
$$
\frac{d}{dt}\|\nabla\rho\|_2^2\leq 2\|\nabla u\|_\infty\|\nabla\rho\|_2^2,
$$
and thus
\begin{align*}
\|\nabla\rho\|_2^2\leq&e^{\int_0^t\|\nabla u\|_\infty ds}\|\nabla\rho_0\|_2^2\leq e^{C\int_0^t\|\nabla u\|_2^{(q-2)/(2q-2)}\|\nabla^2u\|_q^{q/(2q-2)}dt}\|\nabla\rho_0\|_2^2\\
\leq&e^{C\int_0^t(\|\nabla u\|_2+\|\nabla^2\|_q)dt}\|\nabla\rho_0\|_2\leq e^{C(e)}\|\nabla\rho_0\|_2^2.
\end{align*}
By the aid of this estimate, it follows from equation (\ref{1.1}) that
\begin{align*}
\|\rho_t\|_2^2\leq&\|u\nabla\rho\|_2^2\leq C\|u\|_\infty^2\|\nabla\rho\|_2^2\leq\|u\|_2\|\nabla^2u\|_2\|\nabla\rho\|_2^2\leq C(e)\|\nabla\rho_0\|_2^2.
\end{align*}
The proof is complete.
\end{proof}

\section{Local and global strong solutions of Cauchy problem}\label{sec3}

In this section, we prove the global existence and uniqueness of strong solutions to the Cauchy problem of the system (\ref{1.1})--(\ref{1.5}), i.e. we will give the proof of Theorem \ref{thm1.1}. Our strategy is firstly proving the local existence and uniqueness of strong solutions and then extending the local solution to be a global one.

We need the following compactness lemma.

\begin{lemma}\label{lem3.1}
(See Simon \cite{Simon} Corollary 4) Assume that $X, B$ and $Y$ are three Banach spaces, with $X\hookrightarrow\hookrightarrow B\hookrightarrow Y.$ Then the following hold true

(i) If $F$ is a bounded subset of $L^p(0, T; X)$ where $1\leq p<\infty$, and $\frac{\partial F}{\partial t}=\left\{\frac{\partial f}{\partial t}|f\in F\right\}$ is bounded in $L^1(0, T; Y)$. Then $F$ is relatively compact in $L^p(0, T; B)$;

(ii) If $F$ is bounded in $L^\infty(0, T; X)$ and $\frac{\partial F}{\partial t}$ is bounded in $L^r(0, T; Y)$ where $r>1$. Then $F$ is relatively compact in $C([0, T]; B)$.
\end{lemma}

Our local existence and uniqueness result is stated and proven in the following proposition.

\begin{proposition}\label{prop3.1}(Local strong solution) Assume that $0<\underline\rho\leq\rho_0\leq\overline\rho<\infty$, $\nabla\rho_0\in L^2(\mathbb R^2)$, $u_0\in H^2(\mathbb R^2)$ with $\textmd{div}u_0=0$, $\nabla d_0\in H^2(\mathbb R^2)$ with $|d_0|=1$.

Then there exists a time $T$ depending only on $\underline\rho$, $\overline\rho$, $\|u_0\|_{H^2}^2$ and $\|\nabla d_0\|_{H^2}$, such that the system (\ref{1.1})--(\ref{1.5}) complemented with the initial data $(\rho, u, d)|_{t=0}=(\rho_0, u_0, d_0)$ has a unique strong solution $(\rho, u, d, p)$ in $Q_T=\mathbb R^2\times(0, T)$, satisfying
\begin{eqnarray*}
&&\rho\in L^\infty(Q_T),\quad \nabla\rho, \rho_t\in L^\infty(0, T; L^2),\\
&&u\in L^\infty(0, T; H^2)\cap L^2(0, T; W^{2, q}),\quad u_t\in L^\infty(0, T; L^2)\cap L^2(0, T; H^1),\\
&&\nabla d\in L^\infty(0, T; H^2)\cap L^2(0, T; W^{2, q}),\quad d_t\in L^\infty(0, T; H^1)\cap L^2(0, T; H^2),\quad d_{tt}\in L^2(Q_T)\\
&&\nabla p\in L^\infty(0, T; L^2)\cap L^2(0, T; L^q)
\end{eqnarray*}
for $q\in[2,\infty)$.
\end{proposition}

\begin{proof}
Since $u_0\in H^2$, $\textmd{div}u_0=0$, there is a sequence $\{u_{0, R_i}\}_{i=1}^\infty$, with $R_i\uparrow\infty$ and $u_{0, R_i}\in H_0^1(B_{R_i})\cap H^2(B_{R_i})$ and $\textmd{div}u_{0, R_i}=0$, such that $u_{0, R_i}\rightarrow u_0$ in $H^2$. Since $\rho_0\geq\underline\rho>0$, the compatible condition in Lemma \ref{lem1.1} holds true with $p_0=0$. Let $e$ be the constants stated in Lemma \ref{lem1.2} with $u_0$ replaced by $u_{0, R_i}$. Recalling that $u_{0, R_i}\rightarrow u_0$ in $H^2$, then one can easily estimate $e\leq C(\underline\rho, \overline\rho)(1+\|u_0\|_{H^2}+\|\nabla d_0\|_{H^2})^8$.

By Lemma \ref{lem1.1} and Lemma \ref{lem1.2}, there is a time $T$ depending only on $\underline\rho$, $\overline\rho$, $\|u_0\|_{H^2}$, $\|\nabla d_0\|_{H^2}$, such that the system (\ref{1.1})--(\ref{1.5}) has a unique strong solution $(\rho_i, u_i, d_i, p_i)$ in $B_{R_i}(0)\times(0, T)$, complemented with the initial data $(\rho_0, u_{0, R_i}, d_0)$ and the boundary value data $(u, d)|_{\partial\Omega}=(0, d_0)$, and there hold
\begin{align*}
&\sup_{0\leq t\leq T}(\|u_i\|_{H^2(B_{R_i})}^2+\|\nabla d_i\|_{H^2(B_{R_i})}^2+\|\nabla\rho_i\|_{L^2(B_{R_i})}^2+\|\partial_t\rho_i\|_{L^2(B_{R_i})}^2\\
&+\int_0^T(\|\nabla^2u_i\|_{L^q(B_{R_i})}^2+\|\partial_tu_i\|_{H^1(B_{R_i})}^2+\|\nabla p_i\|_{L^q(B_{R_i})}^2+\|\partial_td_i\|_{H^2(B_{R_i})}^2)dt\\
\leq& C(\underline\rho, \overline\rho, \|u_0\|_{H^2}, \|\nabla d_0\|_{H^2}, \|\nabla\rho_0\|_2)
\end{align*}
for $q\in[2,\infty)$. On account of this estimates, using diagonal argument, there is a subsequence, still indexed by $i$, such that
\begin{eqnarray*}
&& u_i\rightarrow u\quad\mbox{ weakly in }L^2(0, T; W^{2,q}(B_R(0))),\quad q\in [2, \infty),\\
&&\partial_tu_i\rightarrow \partial_tu\quad\mbox{ weakly in }L^2(0, T; H^1(B_R(0))),\\
&&d_i\rightarrow d\quad\mbox{ weakly in }L^2(0, T; H^3(B_R(0))),\\
&&\partial_td_i\rightarrow\partial_td\quad\mbox{ weakly in }L^2(0, T; H^2(B_R(0))),\\
&&\rho_i\rightarrow\rho\quad\mbox{ weakly in }L^2(0,T; H^1(B_R(0))),\\
&&\partial\rho_i\rightarrow\partial_t\rho\quad\mbox{ weakly in }L^2(B_R(0)\times(0, T)),\\
&&p_i\rightarrow p\quad\mbox{ weakly in }L^2(0, T; W^{1,q}(B_R(0))),\quad q\in[2,\infty)
\end{eqnarray*}
for any $R>0$, and the following inequality holds true
\begin{align*}
&\sup_{0\leq t\leq T}(\|u\|_{H^2}^2+\|\nabla d\|_{H^2}^2+\|\nabla\rho\|_{L^2}^2+\|\partial_t\rho\|_{L^2}^2\\
&+\int_0^T(\|\nabla^2u\|_{L^q}^2+\|\partial_tu\|_{H^1}^2+\|\nabla p\|_{L^q}^2+\|\partial_td\|_{H^2}^2)dt\\
\leq& C(\underline\rho, \overline\rho, \|u_0\|_{H^2}, \|\nabla d_0\|_{H^2}, \|\nabla\rho_0\|_2)
\end{align*}
By the Lemma \ref{lem3.1}, there holds
\begin{align*}
&u_i\rightarrow u\mbox{ in }L^2(0, T; H^1(B_R(0))),\\
&d_i\rightarrow d\mbox{ in }L^2(0, T; H^2(B_R(0))),\\
&\rho_i\rightarrow\rho\mbox{ in }L^2(B_R(0)\times(0, T))
\end{align*}
for any $R>0$.

On account of all the above convergence, we can take the limit $i\rightarrow\infty$ to conclude that $(\rho, u, d, p)$ is a strong solution to the system (\ref{1.1})--(\ref{1.5}) complemented with the initial data $(\rho_0, u_0, d_0)$. The uniqueness can be proven in the standard way, thus we omit it here. The proof is complete.
\end{proof}

We also need the following rigidity theorem, which grantees the a priori bound on the $L^4$ space time norm of $\nabla d$.

\begin{lemma}\label{lem3.2}(See \cite{LEI})
Let $\varepsilon_0>0$ and $C_0>0$. There exists a positive constant $\delta_0=\delta_0(\varepsilon_0, C_0)$ such that the following holds:

If $d:\mathbb R^2\rightarrow S^2$, $\nabla d\in H^1(\mathbb R^2)$ with $\|\nabla d\|_{L^2(\mathbb R^2)}\leq C_0$ and $d_3\geq\varepsilon_0$, where $d_3$ is the third component of the vector $d$, then
$$
\|\nabla d\|_{L^4(\mathbb R^2)}^4\leq (1-\delta_0)\|\Delta d\|_{L^2(\mathbb R^2)}^2.
$$

Consequently, for such maps the associated harmonic energy is coercive, i.e.,
$$
\|\Delta d+|\nabla d|^2d\|_{L^2(\mathbb R^2)}^2\geq\frac{\delta_0}{2}(\|\Delta d\|_{L^2(\mathbb R^2)}^2+\|\nabla d\|_{L^4(\mathbb R^2)}^4).
$$
\end{lemma}

Now, we can give the proof of Theorem \ref{thm1.1} as follows.

\textbf{Proof of Theorem \ref{thm1.1}.} By Proposition \ref{prop3.1}, there is a unique strong solution $(\rho, u, d, p)$ in $Q_T=\mathbb R^2\times(0, T)$. To show the global existence, we only need to extend such local solution to be a global one. For this purpose, it suffices to show that the quantity
$$
Q(t)=\sup_{0\leq s\leq t}(\|u\|_{H^2}^2+\|\nabla d\|_{H^2}^2+\|\nabla\rho\|_2^2)
$$
is finite for any $t$.

We extend the local strong solution to the maximal time $T_*$. Using the same argument in Lemma \ref{lem1.2}, we can prove that
\begin{eqnarray}
&&\sup_{0\leq s\leq t}(\|u\|_2^2+\|\nabla d\|_2^2)+\int_0^t(\|\nabla u\|_2^2+\|\Delta d+|\nabla d|^2d\|_2^2)ds\leq C,\label{3.1}\\
&&E_1(t)\leq C(t+1)+C\int_0^t(\|\nabla u\|_2^2+\|\nabla d\|_4^4)(E_1(s)+1)ds,\label{3.2}\\
&&E_2(t)\leq C+C\int_0^t(1+E_1^5(s))(1+E_2(s))ds,\label{3.3}\\
&&Q(t)\leq C(\underline\rho, \overline\rho, \|u_0\|_{H^2}, \|\nabla d_0\|_{H^2}, \|\nabla\rho_0\|_2, E_1(t), E_2(t))\label{3.4}
\end{eqnarray}
for any $0\leq t<T_*$, where $C$ is a positive constant depending only on $\underline\rho$, $\overline\rho$, $\|u_0\|_{H^2}$, $\|\nabla d\|_{H^2}$ and $\|\nabla\rho_0\|_2$, $E_1(t)$ and $E_2(t)$ are given by
\begin{eqnarray*}
&&E_1(t)=\sup_{0\leq s\leq t}(\|\nabla u\|_2^2+\|d_t\|_2^2+\|\nabla^2d\|_2^2)+\int_0^t(\|u_t\|_2^2+\|\nabla^2u\|_2^2+\|\nabla d_t\|_2^2+\|\nabla^3d\|_2^2)ds,\\
&&E_2(t)=\sup_{0\leq s\leq t}(\|u_t\|_2^2+\|\nabla d_t\|_2^2)+\int_0^t(\|\nabla u_t\|_2^2+\|d_{tt}\|_2^2+\|\nabla^2d_t\|_2^2)ds.
\end{eqnarray*}

Maximum principle to parabolic equation implies that $d_3\geq\varepsilon_0$, since $d_{0,3}\geq\varepsilon_0$. Thus, we can use Lemma \ref{lem3.2} to deduce
$$
\int_0^t(\|\Delta d\|_2^2+\|\nabla d\|_4^4)ds\leq C\int_0^t\|\Delta d+|\nabla d|^2d\|_2^2ds\leq C.
$$
On account of this inequality, by Gronwall inequality, it follows from (\ref{3.1}) and (\ref{3.2}) that
$$
E_1(t)\leq C(t+1),\qquad t\in [0, T_*).
$$
Combining this with (\ref{3.3}), by Gronwall inequality, we obtain
$$
E_2(t)\leq C(1+t^6)e^{C(1+t^6)},\qquad t\in[0, T_*).
$$
The above two estimates immediately imply, recalling (\ref{3.4}), that $Q(t)$ keeps finite for any finite time. Consequently, Proposition \ref{prop3.1} implies that $T_*=\infty$. The proof is complete.

\section{Global weak solutions}\label{sec4}

In this section, we concern on the proof of global existence of weak solutions to the Cauchy problem of the system (\ref{1.1})--(\ref{1.5}), based on the global existence of strong solutions obtained in the previous section, by using the compactness argument.

As a preparation, we need the following compactness lemma due to Lions \cite{Lions}.

\begin{lemma}\label{lem4.1}
Assume that
\begin{eqnarray*}
&&0\leq\rho^n\leq C\mbox{ a.e. on }\mathbb R^N\times(0, T),\quad\textmd{div}u^n=0\mbox{ a.e. on }\mathbb R^N\times(0, T),\\
&&\partial_t\rho^n+\textmd{div}(\rho^nu^n)=0\mbox{ in }\mathcal D'(\mathbb R^N\times(0, T)),\\
&&\rho_0^n\rightarrow\rho_0\mbox{ in }L^1(B_R(0)),\quad u^n\rightarrow u\mbox{ weakly in }L^2(0, T; H^1(B_R(0)))\mbox{ for all }R\in(0, \infty),\\
&&\frac{u^n}{1+|x|}1_{(\rho^n\geq\delta)}=F_1^n+F_2^n\mbox{ for all }\delta>0, \\
&&F_1^n \mbox{ and }F_2^n \mbox{ are bounded in }L^1(0, T; L^1(\mathbb R^N)) \mbox{ and }L^1(0, T; L^\infty(\mathbb R^N)), \mbox{ respectively. }
\end{eqnarray*}
Then, the following hold

(1) $\rho^n$ converges in $C([0, T]; L^q(B_R(0)))$ for all $q\in[1,\infty)$ and $R\in(0, \infty)$ to the unique solution $\rho$ to
\begin{equation*}
\left\{
\begin{array}{l}
\partial_t\rho+\textmd{div}(\rho u)=0\mbox{ in }\mathcal D'(\mathbb R^N\times(0, T))\\
\rho|_{t=0}=\rho_0
\end{array}
\right.
\end{equation*}
such that
$$
\frac{u}{1+|x|}1_{(\rho\geq\delta)}\in L^1(0, T; L^1(\mathbb R^N))+L^1(0, T; L^\infty(\mathbb R^N)).
$$

(2) We assume in addition that $\rho^n|u^n|^2$ is bounded in $L^\infty(0, T; L^1(\mathbb R^N))$, $u^n$ is bounded in $L^2(0, T; H^2(\mathbb R^N))$, and for some $q\in(1,\infty)$ and $m\geq1$
$$
\left|\left\langle\frac{\partial}{\partial t}(\rho^nu^n),\varphi\right\rangle\right|\leq C\|\varphi\|_{L^q(0, T; W^{m, q}(\mathbb R^N))}
$$
for all $\varphi\in L^q(0, T; W^{m, q}(\mathbb R^N))$ with $\textmd{div}\varphi=0$. Then
$$
\sqrt{\rho^n}u^n\rightarrow\sqrt\rho u\mbox{ in }L^p(0, T; L^r(B_R(0)))
$$
for $2<p<\infty$, $1\leq r<\frac{2Np}{Np-4}$ and $0<R<\infty$.
\end{lemma}

Now, we can give the proof of Theorem \ref{thm1.2} as follows.

\textbf{Proof of Theorem \ref{thm1.2}.} Take a sequence $(\rho_0^{j}, u_0^{j}, d_0^{j})$, such that
\begin{eqnarray*}
&&\frac{1}{j}\leq\rho_0^j\leq\overline\rho+\frac{1}{j}, \quad \nabla\rho_0^j\in L^2(\mathbb R^2),\quad \rho_0^j-\tilde\rho\rightarrow\rho_0-\tilde\rho\mbox{ in }L^q(\mathbb R^2),\\
&&u_0^j\in H^2(\mathbb R^2), \quad\textmd{div}u_0=0, \quad u_0^j\rightarrow u_0\mbox{ in }L^2(\mathbb R^2)\\
&&\nabla d_0^j\in H^2(\mathbb R^2),\quad|d_0^j|=1, \quad d_{0, 3}^j\geq\frac{\varepsilon_0}{2},\quad\nabla d_0^j\rightarrow\nabla d_0\mbox{ in }L^2(\mathbb R^2).
\end{eqnarray*}
By Theorem \ref{thm1.1}, for each $j$ there is a unique global strong solution $(\rho^j, u^j, d^j)$. It follows from (\ref{1.1}) and our assumption that
\begin{equation}\label{4.1}
\|\rho^j(t)-\tilde\rho\|_q=\|\rho_0^j-\tilde\rho\|_q\rightarrow\|\rho_0-\tilde\rho\|_q.
\end{equation}
Multiplying (\ref{1.2}) by $u^j$, (\ref{1.4}) by $-\Delta d^j$, integrating over $\mathbb R^2$, using (\ref{1.1}) and $|d^j|=1$, summing the resulting identities up yields
\begin{align*}
&\frac{d}{dt}\int_{\mathbb R^2}\left(\frac{\rho^j}{2}|u^j|^2+\frac{|\nabla d^j|^2}{2}\right)dx+\int_{\mathbb R^2}(|\nabla u^j|^2+|\Delta d^j+|\nabla d^j|^2d^j|)dx
=\int_{\mathbb R^2}\left(\frac{\rho_0^j}{2}|u_0^j|^2+\frac{|\nabla d_0^j|^2}{2}\right)dx.
\end{align*}
By Lemma \ref{lem3.1}, noticing that $d_3^j\geq\frac{\varepsilon_0}{2}$, we deduce from the above identity that
\begin{align}
&\sup_{0\leq t\leq T}(\|\sqrt{\rho^j}u^j\|_2^2+\|\nabla d^j\|_2^2)+\int_0^T(\|\nabla u^j\|_2^2+\|\nabla^2 d^j\|_2^2)dt\leq C(\|\sqrt{\rho_0}u_0\|_2^2+\|\nabla d_0\|_2^2). \label{4.2}
\end{align}
By (\ref{4.1}) and (\ref{4.2}), it follows from Gagliado-Nirenberg inequality that
\begin{align*}
&\tilde\rho\int_{\mathbb R^2}|u^j|^2dx\leq\int_{\mathbb R^2}(|\rho^j-\tilde\rho||u^j|^2+\rho^j|u^j|^2)dx\\
\leq&C(\|\sqrt{\rho_0}u_0\|_2^2+\|\nabla d_0\|_2^2)+\int_{\mathbb R^2}|\rho^j-\tilde\rho||u^j|^2dx\\
\leq&C(\|\sqrt{\rho_0}u_0\|_2^2+\|\nabla d_0\|_2^2)+\left(\int_{\mathbb R^2}|\rho^j-\tilde\rho|^qdx\right)^{1/q}\left(\int_{\mathbb R^2}|u^j|^{2q/(q-1)}dx\right)^{(q-1)/q}\\
\leq&C(\|\sqrt{\rho_0}u_0\|_2^2+\|\nabla d_0\|_2^2)+C\|\rho_0-\tilde\rho\|_q\left(\int_{\mathbb R^2}|u^j|^{2}dx\right)^{(q-1)/q}\left(\int_{\mathbb R^2}|\nabla u^j|^{2}dx\right)^{1/q}\\
\leq&\frac{\tilde\rho}{2}\int_{\mathbb R^2}|u^j|^{2}dx+C(\|\sqrt{\rho_0}u_0\|_2^2+\|\nabla d_0\|_2^2)+C\|\rho_0-\tilde\rho\|_q^q\int_{\mathbb R^2}|\nabla u^j|^{2}dx,
\end{align*}
and thus
$$
\int_{\mathbb R^2}|u^j|^2dx\leq C(\|\sqrt{\rho_0}u_0\|_2^2+\|\nabla d_0\|_2^2)+C\|\rho_0-\tilde\rho\|_q^q\int_{\mathbb R^2}|\nabla u^j|^{2}dx,
$$
which, by (\ref{4.2}), there holds
\begin{equation}\label{4.3}
\int_0^T\int_{\mathbb R^2}|u^j|^2dx\leq C(\|\sqrt{\rho_0}u_0\|_2^2+\|\nabla d_0\|_2^2+\|\rho_0-\tilde\rho\|_q^q)(T+1).
\end{equation}

By Ladyzhenskaya inequality and Gagliado-Nirenberg inequalities, it follows from (\ref{4.2}) and (\ref{4.3}) that
\begin{eqnarray}
&&\|\nabla d^j\|_{L^4(Q_T)}^4\leq C(\|\sqrt{\rho_0}u_0\|_2^2+\|\nabla d_0\|_2^2)^2,\label{4.4}\\
&&\|u^j\|_{L^2(0, T; L^q(\mathbb R^2))}^2\leq C\|u^j\|_{L^2(0, T; H^1(\mathbb R^2))}^2\nonumber\\
&&\leq C(T+1)(\|\sqrt{\rho_0}u_0\|_2^2+\|\nabla d_0\|_2^2+\|\rho_0-\tilde\rho\|_q^q),\quad q\in[2,\infty).\label{4.5}
\end{eqnarray}
On account of these two inequalities, there holds
$$
\|u^j\cdot\nabla d\|_{L^{4/3}(0,T; L^r(\mathbb R^2))}\leq C(T+1)(\|\sqrt{\rho_0}u_0\|_2^2+\|\nabla d_0\|_2^2+\|\rho_0-\tilde\rho\|_q^q),\quad r\in[4/3, 4),
$$
and thus, it follows from equation (\ref{1.4}) that
$$
\|\partial_t d^j\|_{L^{4/3}(0,T; L^2(\mathbb R^2))}\leq C(T+1)(\|\sqrt{\rho_0}u_0\|_2^2+\|\nabla d_0\|_2^2+\|\rho_0-\tilde\rho\|_q^q+1).
$$
By Lemma \ref{lem3.1}, there is a subsequence, still indexed by $i$, such that
$$
d^i\rightarrow d\qquad\mbox{ in }L^2(0, T; H^1(B_R(0)))\mbox{ for all }R\in(0, \infty).
$$
By (\ref{4.2}) and (\ref{4.5}), there holds
\begin{equation}\label{4.6}
\|\rho^ju^j\otimes u^j\|_{L^2(0, T; L^r(\mathbb R^2))}\leq C\|\sqrt{\rho^j}u^j\|_{L^\infty(0, T; L^2(\mathbb R^2))}\|u^j\|_{L^2(0,T; L^{2r/(2-r)}(\mathbb R^2))}\leq C,\quad\mbox{ for }r\in[1,2).
\end{equation}
By the aid of (\ref{4.4}) and (\ref{4.6}), it follows from (\ref{1.2}) that
\begin{align*}
\left|\left\langle\frac{\partial}{\partial t}(\rho^ju^j),\varphi\right\rangle\right|=&\left|\int_0^T\int_{\mathbb R^2}(\nabla d^j\otimes \nabla d^j+\rho^ju^j\otimes u^j-\nabla u^j):\nabla\varphi dxdt\right|\\
\leq&C(\|\nabla d^j\|_{L^4(Q_T)}^2\|\nabla\varphi\|_{L^2(Q_T)}+\|\rho^ju^j\otimes u^j\|_{L^2(0, T; L^{q/(q-1)(\mathbb R^2)})}\|\nabla\varphi\|_{L^2(0,T; L^q(\mathbb R^2))}\\
&+\|\nabla u^j\|_{L^2(Q_T)}\|\nabla\varphi\|_{L^2(Q_T)})\\
\leq&C(\|\nabla\varphi\|_{L^2(Q_T)}+\|\nabla\varphi\|_{L^2(0,T; L^q(\mathbb R^2))})\leq C\|\varphi\|_{L^2(0,T;H^2(\mathbb R^2))}
\end{align*}
for any $q\in(2,\infty)$ and $\varphi\in L^2(0, T; H^2)$ with $\textmd{div}\varphi=0$.
Now, we can apply Lemma \ref{lem4.1} to deduce that $\sqrt{\rho^j}u^j\rightarrow\sqrt\rho u$ in $L^p(0, T; L^r(B_R(0))$ with $2<p<\infty$ and $1\leq r\leq\frac{4p}{2p-4}$ and $R\in(0, \infty)$. Hence, we can take the limit $i\rightarrow\infty$ to conclude that $(\rho, u, d)$ is a weak solution to the system (\ref{1.1})--(\ref{1.5}) with the initial data $(\rho_0, u_0, d_0)$. The proof is complete.

\end{document}